\documentclass[10pt]{amsart}

\usepackage[french, english]{babel}
\usepackage[T1]{fontenc}
\usepackage{amsmath,amssymb,stmaryrd}
\usepackage[all]{xy}
\usepackage{graphicx}
\usepackage{fancyhdr}

 \theoremstyle{plain}
 \newtheorem{lemma}{Lemma}[section]

\newtheorem{theorem}[lemma]
{Theorem}
\newtheorem{corollary}[lemma]
{Corollary}

\newtheorem{prop}[lemma]{Proposition}

\theoremstyle{definition}
\newtheorem{definition}[lemma]{Definition }

\newtheorem{rmk}[lemma]{Remark}

\newcommand{\lgw}{\longrightarrow}

\newcommand{\s}{\sigma}

\newcommand{\md}{\text{mod}}

\newcommand{\wdh}{\widehat}

\renewcommand{\L}{\mathbb{L}}
\newcommand{\m}{\mathfrak{m}}

\newcommand{\Z}{\mathbb{Z}}

\renewcommand{\k}{\Bbbk}
\renewcommand{\dim}{\text{dim}}
\newcommand{\td}{\text{tr.deg}}

\newcommand{\K}{\mathbb{K}}

\newcommand{\gr}{\text{gr}}
\newcommand{\N}{\mathbb{N}}

\newcommand{\h}{\Phi}

\newcommand{\ch}{\text{char}}
\newcommand{\he}{\text{ht}}
\newcommand{\Q}{\mathbb{Q}}

\newcommand{\rr}{\text{rat.rk}}
\newcommand{\p}{\textsl{p}}
\newcommand{\q}{\textsl{q}}

\newcommand{\ddo}{,\cdots,}

\newcommand{\g}{\Gamma}

\newcommand{\lt}{\text{length}}
\renewcommand{\phi}{\varphi}

\begin{document}

\title{ The analogue of Izumi's Theorem for  Abhyankar valuations}
\author{by G. Rond and M. Spivakovsky}

\address{Institut de Math\'ematiques de Marseille, 
Universit\'e d'Aix-Marseille, Campus de Luminy, Case 907,
13288 Marseille Cedex 9}
\email{guillaume.rond@univ-amu.fr}
\address {Institut de Math\'ematiques de Toulouse
Universit\'e Paul Sabatier,
118 route de Narbonne,
31062 Toulouse Cedex 9 ,
France}
\email{mark.spivakovsky@math.univ-toulouse.fr}

\dedicatory{Dedicated to the memory of Shreeram Abhyankar and David Rees.}

\begin{abstract}
A well known theorem of Shuzo Izumi, strengthened by David Rees, asserts that all the divisorial valuations centered in an analytically irreducible local noetherian ring $(R,\m)$ are linearly comparable to each other. This is equivalent to saying that any divisorial valuation $\nu$ centered in $R$ is linearly comparable to the $\m$-adic order. In the present paper we generalize this theorem to the case of Abhyankar valuations $\nu$ with archimedian value semigroup $\Phi$. Indeed, we prove that in a certain sense linear equivalence of topologies characterizes Abhyankar valuations with archimedian semigroups, centered in analytically irreducible local noetherian rings. In other words, saying that $R$ is analytically irreducible, $\nu$ is Abhyankar and $\Phi$ is archimedian is equivalent to linear equivalence of topologies plus another condition called weak noetherianity of the graded algebra $\gr_\nu R$.

We give some applications of Izumi's theorem and of Lemma \ref{lemma_norm}, which is a crucial step in our proof of the main theorem. We show that some of the classical results on equivalence of topologies in noetherian rings can be strengthened to include linear equivalence of topologies. We also prove a new comparison result between the $\m$-adic topology and the topology defined by the symbolic powers of an arbitrary ideal. 
\end{abstract}

\subjclass[2000]{Primary: 13A18. Secondary: 13A15, 16W60}

\maketitle

\section{Introduction}

Let ($R$, $\m$, $\k $) be a local noetherian domain with the
maximal ideal $\m$ and residue field $\k $. Let $\K$ denote the field of fractions of $R$.
Consider a valuation $\nu: \K^*\twoheadrightarrow \g $  of $\K$ with value group
$\g$. We denote by $R_{\nu}$ its valuation ring and by $\m_{\nu}$ its maximal ideal.
\begin{definition} We say that $\nu$ is centered in $R$ if $ \nu$ is non-negative on $R$  and strictly
positive on $\m$.
\end{definition}
Consider a valuation $\nu :R\longrightarrow \g$, centered in $R$.  Then $\m_{\nu}\cap R=\m$; 
thus  $\k$ is a subfield of $\k_{\nu}:= \frac{R_{\nu}}{\m_{\nu}}$.
\begin{definition} We say that $\nu$ is a {\it divisorial valuation} if its value group $\g=\Z$ and
$\td_{\k}\k_{\nu}=\dim\ R-1$.
\end{definition}
The purpose of this paper is to generalize the following theorem of Shuzo Izumi and David Rees (often called Izumi's Theorem for short) to a larger class of valuations than just the divisorial ones.
\begin{theorem}\label{Izumi}\cite{8}, \cite{Re3}
Let $R$  be an analytically irreducible local domain. Then for any two divisorial valuations $\nu$ and $\nu'$, centered in $R$, there exists a constant $k>0$ such that
$$\nu(f)\leq  k\,\nu'(f)\ \ \ \ \forall f\in R.$$
\end{theorem}
This result has played a central role in the study of ideal-adic topologies and other questions about commutative rings during the last decades.

To highlight its applications, we start with some basic definitions and then recall two related theorems due to David Rees.

Let $R$  be a commutative noetherian ring and  $I$  an ideal of $R$. For an element $f\in R$
the $I$-order of $f$ is defined to be
$$
I(f):=\max\{n\in\N\mid f\in I^n\}.
$$
This function takes its values in $\N\cup\{\infty\}$ and $I(f)=\infty$ if and only if $f\in I^n$ for all $n\in\N$. Moreover it is easy to see that $I(fg)\geq I(f)+I(g)$ for all $f$, $g\in R$ since $I$ is an ideal.\\

Next we introduce a more invariant notion, defined by David Rees and Pierre Samuel, namely the
{\it reduced} order:
$$
\bar I(f):=\lim_{n\to\infty}\frac{I(f^n)}n.
$$
A priori, it is not obvious that $\bar I(f)$ is a rational number, or even
finite. The fact that $\bar I(f)$ is always rational in a noetherian domain
$R$  is a consequence of the following theorem of D. Rees:
\begin{theorem} \label{Rees_val}\cite{Re0}\cite{Re1} For any ideal $I$ in a noetherian domain $R$  there exists a
unique finite set of valuations $\{\nu_i\}_{1\leq  i\leq  r}$ of $R$  (with values in
$\Z$) such that
$$
\bar I(f)=\min_{1\leq  i\leq  r}\frac{\nu_i(f)}{\nu_i(I)}
$$
and this representation is irredundant. These valuations $\nu_i$ are called the Rees valuations of the ideal $I$.
\end{theorem}
\begin{rmk} This Theorem has been stated and proved by Rees assuming only that $R$ is noetherian (not necessarily a domain); here we restrict ourselves to the case of domains in order to simplify the exposition.
\end{rmk}
\begin{rmk} In view of Theorem \ref{Rees_val}, Theorem \ref{Izumi} can be reformulated as follows: every divisorial valuation $\nu$ is linearly equivalent to the $\m$-adic order. This means that there exists a positive integer $k$ such that $\nu(f)<k\m(f)$ for all $f\in R$.  
\end{rmk}
In the case $R$  is a local domain and $I=\m$, the valuations $\nu_i$ appearing in the statement of Theorem \ref{Rees_val} are divisorial valuations centered in $R$.\\
From the definitions we see easily that $\bar I(f)\geq I(f)$ for all $f$. We will need the following result of D. Rees in order to derive some corollaries of Theorem \ref{Izumi} about ideal-adic topologies in noetherian rings:

\begin{theorem}\label{Rees_th}\cite{Re2} Let $R$  be an analytically unramified local ring. Then there exists a constant $C>0$ such that
$$\bar I(f)\leq  I(f)+C\ \ \ \ \forall f\in R.$$
\end{theorem}

The goal of this paper is to generalize Theorem \ref{Izumi} to a larger class of valuations: the Abhyankar valuations whose semigroup is archimedean.\\
\\

Let us begin with some definitions.
Let $\K$ denote the field of fractions of a local domain $R$.
Consider a valuation
$\nu: \K^*\twoheadrightarrow \g $  of $\K$ with value group
$\g$, centered in $R$. Let $R_{\nu}$ denote the valuation ring of $\nu$ and
$\m_{\nu}$ the maximal ideal of $R_{\nu}$. Since $\nu$ is centered
in $R$, we have a natural injection $\k \subset \frac {R_{\nu}}{\m_{\nu}}$. The three basic invariants associated with $\nu$ are
\begin{align} \td_{\k }\ \nu &:= \td\  \left(\frac{R_{\nu}}{\m_{\nu}}/\k \right)\\
\rr\ \nu &:=\dim_{\Q}(\g\otimes_{\Z}\Q)\\
\text{rk}\ \nu&:=\dim R_{\nu},
\end{align}
where ``dimension'' means the Krull dimension.
In 1956 S. Abhyankar proved that
\begin{equation}\label{1.1}\rr\ \nu+\td_{\k }\ \nu\leq  \dim R\end{equation}
(cf. \cite{1}, Theorem 1, p. 330).  This inequality is called the \textit{Abhyankar Inequality}. An {\it Abhyankar valuation} is a valuation centered in $R$  such that the above inequality is an equality. Any divisorial valuation centered in $R$  satisfies $\rr\ \nu=1$ and $\td_{\k}\ \nu=\dim R-1$; hence it is an Abhyankar valuation.\\
Let
$$\Phi := \nu(R\setminus\{0\})\subset\g.
$$
Then $\Phi$ is an ordered semigroup contained in $\g$. Since $R$  is
noetherian, $\Phi$ is a well-ordered set. For $\alpha\in\Phi$, let
$$
\p_{\alpha}\ := \{x\in R\mid \nu(x)\ge\alpha\}
$$
and
$$\p_{\alpha+}\ := \{x\in R \mid \nu(x)>\alpha\}
$$
(here we adopt the convention that $\nu(0)>\alpha$ for any
$\alpha\in\g$). Ideals of $R$  which are contractions to $R$  of ideals
in $R_{\nu}$ are called $\nu${\it-ideals}. All of the $\p_{\alpha}$ and
$\p_{\alpha+}$ are $\nu$-ideals and
$\{\p_{\alpha}\}_{\alpha\in\Phi}$ is the complete list of
$\nu$-ideals in $R$. Specifying all the $\nu$-ideals in $R$  is equivalent to
specifying $\nu$ (see \cite{26}, Appendix 3). The following is a characterization of $\nu$-ideals: an ideal $I\subset R$ is a $\nu$-ideal if and only if, for any elements $a\in R$, $b\in I$ such that $\nu(a)\geq\nu(b)$ we have $a\in I$.

We associate to $\nu$ the following graded algebra:
$$
\gr_{\nu}R:=\bigoplus_{\alpha\in\Phi}
\frac{\p_{\alpha}}{\p_{\alpha+}}\ .
$$

\begin{definition} We say that $\h$ is \emph{archimedian} if for any
$\alpha,\beta\in\h$, $\alpha\ne0$, there exists $r\in\N$ such that
$r\alpha>\beta$.
\end{definition}
This is equivalent to saying that every $\nu$-ideal in $R$  is $\m$-primary
and weaker than saying that $\text{rk}\ \nu=1$.

For $l\in\N$, let $Q_l$ denote the $\nu$-ideal
\begin{equation}\label{1.2}
Q_l:=\{x\in R\mid\nu(x)\geq l\nu(\m)\}.
\end{equation}
Of course, $\m^l\subset Q_l$ for all $l\in\N$.
\begin{definition} We say that the $\nu$-adic and the $\m$-adic topologies are
linearly equivalent if there exists $r\in\N$ such that
$$
Q_{rl}\subset\m^l
$$
for all $l\in\N$.
\end{definition}
Thus Theorem \ref{Izumi} is equivalent to saying that  for any divisorial valuation $\nu$ of $R$  the $\nu$-adic topology is linearly equivalent to the $\m$-adic topology. \\
Let $\displaystyle A=\bigoplus\limits_{\alpha\in\h}A_\alpha$ be a $\h$-graded $\k$-algebra. Assume
that $A_0=\k$. By abuse of notation, let us denote
$$
1:=\nu(\m)
$$
and for $l\in\N$
$$
l:=l\cdot1.
$$
\begin{definition} We say that $A$ is {\rm weakly noetherian of dimension }$d$ if
$A$ contains $d$ algebraically independent elements and the function
$$
F(l)=\sum_{0\le\alpha\leq  l}\dim_\k A_\alpha
$$
is bounded above by a polynomial in $l$ of degree $d$.\end{definition}
If $A$ is weakly noetherian then $\dim_{\k} A_{\alpha}<\infty$ for
all $\alpha\in\h$. We now state the main theorem of this paper.
\begin{theorem}\label{main_th} Let $(R,\m,\k)$ be a local noetherian domain with field of
fractions $\K$. Let $\nu$ be a valuation of $\K$ centered in $R$  with value
semigroup $\h$. Then the following two conditions are equivalent:
\begin{enumerate}
\item $R$  is analytically irreducible, $\h$ is archimedian and
\begin{equation}\label{1.3}
\rr\ \nu+\td_\k\nu=\dim\ R\end{equation}

\item the $\nu$-adic and the $\m$-adic topologies are linearly equivalent
and $\gr_\nu R$ is weakly noetherian.
\end{enumerate}
\end{theorem}
This theorem is proved in Part \ref{S2}. Part \ref{S3} is devoted to applications of Izumi's Theorem and of Lemma \ref{lemma_norm}.\\

The most difficult part in the proof of Theorem \ref{main_th} is the implication (1)$\Longrightarrow$ "linear equivalence between the $\nu$-adic and the $\m$-adic topologies". Let us mention that this implication has been proved by D. Cutkosky \cite{2} in the case $R$ contains a characteristic zero field. More precisely he proved in this case that an Abhyankar valuation is quasi-monomial after a sequence of blow-ups. Thus, by Theorem \ref{Izumi}, the $\nu$-adic and the $\m$-adic topologies are linearly equivalent. This implication has also been proved in the case when $(R,\m)$ is a complete local ring containing a field of positive characteristic and $\nu$ is the composition of a morphism $R\lgw S$, where $(S, \m_0)$ is a regular complete local ring with $\frac{S}{\m_0}\simeq\frac{R}{\m}$, and of the $\m_0$-adic valuation of $S$ (see \cite{R}).\\
\\
Theorem \ref{main_th} has been announced by the second author and a proof was sketched in \cite{Sp1} without details but the entire proof was never published. The proof presented here follows the sketched proof announced in \cite{Sp1}.


\subsection{Conventions}

 Let $R$  be an integral domain with field of
fractions $\K$. Let $\nu$ be a valuation of $\K$ with value group $\Gamma$.

The notation $R_{\nu}$ will stand for the valuation ring of $\nu$ (the ring of all
the elements of $\K$ whose values are non-negative) and $\m_{\nu}$
the maximal ideal of $R_{\nu}$ (the elements with strictly positive
values).
 If $\p\subset R$ is an ideal, we
put
$$\nu(\p):=\ \min\  \{\nu(x)\mid x\in \p\}.$$
If $R$  is noetherian, this minimum is always achieved.\\
If $x$ is an element of $R$, $\bar x$ will always
mean the natural image of $x$ in $\gr_\nu R$.

We freely use the multi-index notation: if $x=(x_1\ddo x_n)$,
$\alpha=(\alpha_1\ddo \alpha_n)\in\N_0^n$, then $x^{\alpha}$ stands for
$\prod\limits_{i=1}^nx_i^{\alpha_i}$. The syumbol $|\alpha|$ will stand for $\sum\limits_{i=1}^n\alpha_i$.

Let $R$  be a regular local ring with regular system of parameters $x=(x_1\ddo x_n)$. A valuation $\nu$, centered at $R$, is said to be \emph{monomial} with respect to $x$ if all the $\nu$-ideals of $R$  are generated by monomials in $x$.


\section{Proof of Theorem \ref{main_th}}\label{S2}

The purpose of this section is to prove Theorem \ref{main_th}.
We start with a few remarks.

\vskip0.1in
\noindent
\begin{rmk}\label{2impliesanirred} Suppose the $\m$-adic and the $\nu$-adic topologies in $R$  are
equivalent. Then $\h$ is archimedian. Indeed, to say that $\h$ is not
archimedian is equivalent to saying that there exists a $\nu$-ideal $\p$ in
$R$ which is not $\m$-primary. But
then $\p$ is open in the $\nu$-adic
topology, but not in the $\m$-adic topology, which is a contradiction.

Moreover, the equivalence of topologies implies that $R$  is analytically
irreducible. Indeed, let $\{a_n\}$, $\{b_n\}$ be two Cauchy sequences for the $\m$-adic topology in
$R$  such that
\begin{align}
\lim_{n\to\infty}a_n&\ne0\label{eq:anto0}\\
\lim_{n\to\infty}b_n&\ne0\label{eq:bnto0},
\end{align}

but
$$
\lim_{n\to\infty}a_nb_n=0.
$$
By the equivalence of topologies, (\ref{eq:anto0}) and (\ref{eq:bnto0}), $\nu(a_n)$ and $\nu(b_n)$
are independent of $n$ for $n\gg0$, hence so is $\nu(a_nb_n)$, which contradicts the
fact that
$$
\lim_{n\to\infty}a_nb_n=0
$$
in the $\nu$-adic topology.\end{rmk}
\vskip0.1in\noindent
\begin{rmk} The fact that $\h$ is archimedian together with the equality (\ref{1.3})
implies that the $\m$-adic and the $\nu$-adic topologies are equivalent (see
\cite{25}, pp. 63--64).\end{rmk}
\begin{lemma}\label{lem2.3} Let $R$ be any local domain whatsoever and $\nu$ any valuation of
the field of fractions, centered in $R$. Let
$x_1\ddo x_r$ be elements of $R$ such that $\{\nu(x_i)\}_{1\leq  i\leq  r}$ are
linearly independent over $\Z$. Let $y_1\ddo y_t$ be elements of $R_{\nu}$
such that the natural images $\bar y_i$ of the $y_i$ in $\frac{R_{\nu}}{\m_{\nu}}$ are
algebraically independent over $\k$. Assume that there exists a monomial $b=x^\omega$ such that
$by_i\in R$ for all $i$, $1\leq  i\leq  t$. Then the natural images of $x_1\ddo
x_r, by_1\ddo by_t$ in $\gr_\nu R$ are algebraically independent over $\k$.
\end{lemma}
\begin{proof} The algebra $\gr_\nu R$ is an integral domain, on which $\nu$ induces a natural
valuation, which we shall also denote by $\nu$. Let
$$
z_i:=by_i,\qquad 1\leq  i\leq  s.
$$
Consider an algebraic
relation
\begin{equation}\label{2.1}
\sum_{\alpha,\beta}\bar c_{\alpha,\beta}\bar x^\alpha
\bar z^\beta=0,\qquad\alpha\in\N_0^r,\beta\in\N_0^s, \bar
c_{\alpha,\beta}\in\k,\end{equation}

where the $\bar c_{\alpha,\beta}$ are not all zero. Here $\bar x_i$, $\bar z_i$
denote the
natural images in $\gr_\nu R$ of $x_i$ and $z_i$, respectively. We may
assume that (\ref{2.1}) is homogeneous with respect to $\nu$. At least two of the
$\bar c_{\alpha,\beta}$ must be non-zero. Take a pair
$(\alpha,\beta),(\gamma,\delta)\in\N_0^{s+r}$ such that
\begin{align}
\bar c_{\alpha,\beta}&\neq0\\
\bar c_{\gamma,\delta}&\neq0.
\end{align}

Then
$$
\bar x^\alpha\bar b^{|\beta|}=\bar x^\gamma\bar b^{|\delta|}
$$
otherwise the equality
$$
\nu(\bar c_{\alpha,\beta}\bar x^\alpha\bar z^\beta)=\nu(\bar
c_{\gamma,\delta}\bar x^\gamma\bar z^\delta)
$$
would give a rational dependence between the $\nu(x_i)$. This implies that (\ref{2.1}) can be rewritten in the form
$$
\bar x^{\lambda}\sum_{\alpha,\beta}\bar c_{\alpha,\beta}\bar y^{\beta}=0
$$
and hence
\begin{equation}\label{2.2}
\sum_{\alpha,\beta}\bar c_{\alpha,\beta}\bar y^{\beta}=0,\end{equation}
since $\gr_{\nu} R$ is an integral domain. But this contradicts the choice of
the $y_i$ and the Lemma is proved.
\end{proof}
\begin{corollary}\label{cor2.4} Under the assumptions of Lemma \ref{lem2.3}, let

\begin{align}
r:&=\rr\ \nu\\
t:&=\td_\k\nu.
\end{align}

Then $\td_\k\gr_\nu R=r+t$. Here we allow the possibility for both sides to
be infinite. In particular, if $\gr_\nu R$ is weakly
noetherian, its dimension
must be $r+t$.
\end{corollary}
\begin{proof} We work under the assumption that $r$  and $t$ are finite and
leave the general case as an easy exercise. Let $y_1\ddo y_t$ be a maximal
set of elements of $R_{\nu}$ such that
the $\bar y_i$ are algebraically independent over $\k$. Let $x_1\in R$ be
any element of strictly positive value such that $x_1y_i\in R$ for all $i$, $1\leq  i\leq  t$. Choose
$x_2\ddo x_r$ in such a way that $\{\nu(x_i)\}_{1\leq i\leq r}$ form a basis
for $\g\otimes_{\Z}\Q$ over $\Q$. By Lemma \ref{lem2.3}, $\bar x_1\ddo\bar
x_r,\bar x_1\bar y_1\ddo
\bar x_1\bar y_t$ are algebraically independent over $\k$ in $\gr_\nu R$.
Hence,
$$
\td_\k\gr_\nu R\geq r+t.
$$
On the other hand, take any $z\in R$. By the choice of the $x_i$, there
exist $l\in\N$, $c_i\in\Z$ such that
$$
\nu\left(z^l\right)=\nu\left(\prod_{i=1}^rx_i^{c_i}\right).
$$
By the choice of the $y_i$, $\frac{\bar z^l}{\prod_{i=1}^r\bar x_i^{c_i}}$
is algebraic over $\k(\bar y_1\ddo\bar y_t)$. Writing down the algebraic
dependence relation and clearing denominators, we get that $\bar z$ is
algebraic over $\k[\bar x_1\ddo\bar x_r,\bar x_1\bar y_1\ddo\ \bar x_1\bar
y_t]$ and the proof is complete.
\end{proof}
Proof of Theorem \ref{main_th} (2)$\implies$(1). By Remark \ref{2impliesanirred} we only have to prove
that
$$
\rr\ \nu+\td_\k\nu=\dim\ R.
$$
Let
$$
d:=\rr\ \nu+\td_\k\nu
$$
and for $l\in\N$
$$
N(l):=\lt\frac R{Q_l}
$$
where $Q_l$ is as in (\ref{1.2}). By Corollary \ref{cor2.4}, $\gr_\nu R$ is weakly
noetherian of dimension $d$. Hence $N(l)$ is bounded above by a polynomial
in $l$ of degree $d$. By the linear equivalence of topologies, there exists
$r\in\N$ such that
$$
Q_{rl}\subset\m^l\qquad\text{for all }l\in\N.
$$
Hence
$$
\lt\frac R{\m^l}\le\lt\frac R{Q_{rl}},
$$
and so $\lt\frac R{\m^l}$ is bounded above by a polynomial of degree $d$ in
$l$. Therefore $\dim\ R\leq d$, hence $\dim\ R=d$ by Abhyankar's Inequality \eqref{1.1}.

\noindent(1)$\implies$(2). Let $d$ be as above. By Corollary \ref{cor2.4}, $\gr_\nu
R$ contains $d$ algebraically independent elements over $\k$. For $l\in\mathbb
N$, $\m^l\subset Q_l$, so that
$$
\lt\frac R{\m^l}\ge\lt\frac R{Q_l}.
$$
Therefore $\lt\frac R{Q_l}$ is bounded above by a polynomial in $l$ of
degree
$d$ and $\gr_\nu R$ is weakly noetherian.

We have now come to the hard part of the Theorem: proving the linear
equivalence of topologies.

Let $\wdh R$ be the $\m$-adic completion of $R$. Since $R$ is analytically
irreducible and $\h$ archimedian, there exists a unique extension $\wdh\nu$
of $\nu$ to $\wdh R$ (see \cite{25}, pp. 63--64). Moreover, $\wdh\nu$ has the same
value group as $\nu$ and $\frac{R_{\wdh\nu}}{\m_{\wdh\nu}}=\frac{R_{\nu}}{\m_{\nu}}$.
Since
$$
\m^n\wdh R=(\m\wdh R)^n,
$$
it is sufficient to prove that the $\wdh\nu$-adic and the $\m\wdh R$-adic
topologies are linearly equivalent in $\wdh R$. Thus we may assume that $R$
is complete.
\vskip0.1in\noindent
\textit{Claim}. There exists a system of parameters $(x_1\ddo x_d)$ of $R$ such
that $\bar x_1\ddo\bar x_d$ are algebraically independent in $\gr_\nu R$
over
$\k$, and if char$(R)=0$ and char$(\k)=p>0$ we have $x_1=p$.
\vskip0.1in\noindent
\begin{proof}[Proof of Claim] We construct the $x_i$ recursively. First of all we choose any non-zero element $x_1$ in $\m$ except in the case char$(R)=0$ and char$(\k)=p>0$ where we choose $x_1:=p$. Assume that we already constructed elements
$$
x_1\ddo x_i\in\m
$$
such that $\bar x_1\ddo\bar x_i$ are algebraically independent over $\k$ and $\he(x_1\ddo x_i)=i<d$. By Corollary \ref{cor2.4} there exists  $y\in R$ such that $\bar y$
is transcendental over $\k[\bar x_1\ddo\bar x_i]$ in $\gr_\nu R$. Let $P_1\ddo P_s$ denote the minimal prime ideals of 
$(x_1\ddo x_i)R$. Renumbering the $P_j$, if necessary, we may assume that there exists $j\in\{0\ddo s\}$ such that $y\in P_l$, 
$1\leq l\leq j$ and $y\notin P_l$, $j<l\leq s$.

Take any $z\in\bigcap\limits_{l=j+1}^sP_l\setminus\bigcup\limits_{l=1}^jP_l$ (where we take $\bigcup\limits_{l=1}^jP_l=\emptyset$ if 
$j=0$ and $\bigcap\limits_{l=j+1}^sP_l=R$ if $j=s$), and let
$$
x_{i+1}:=y+z^N,
$$
where $N$ is an integer such that $N\nu(z)>\nu(y)$. We have constructed elements
$$
x_1\ddo x_i,x_{i+1}\in\m
$$
such that $\bar x_1\ddo\bar x_{i+1}$ are algebraically independent over $\K$ and
$$
\he(x_1\ddo x_{i+1})=i+1.
$$
For $i=d$ we obtain the desired system of parameters $(x_1\ddo x_d)$. The Claim is
proved.\end{proof}
\vskip0.1in
If $R$ is not equicharacteristic, it contains a complete
non-equicharacteristic Dedekind domain $W$ (cf. \cite{9}, Theorem 84) whose maximal ideal is generated by $p=x_1$. Since $R$ is an integral domain, $R$ is
finite over $\k[[x_1\ddo x_d]]$ or $W[[x_2\ddo x_d]]$, depending on whether
$R$ is equicharacteristic or not (cf. \cite{9}, Theorem 84). Let
$$
S:=\k[[x_1\ddo x_d]]\quad\text{or}\quad W[[x_2\ddo x_d]],
$$
depending on which of the two cases we are in. Let $\L$  denote the field of
fractions of $S$ and $\m_0:=\m\cap S$. Let $t_1\ddo t_n$ be a system of
generators of $R$ as an $S$-algebra. Let $T_1\ddo T_n$ be independent
variables and write
$$
R=\frac{S[T_1\ddo T_n]}\p,
$$
where $\p$ is the kernel of the natural map $S[T]\rightarrow R$, given by
$T_i\rightarrow t_i$, $1\leq i\leq n$. Let
$$
\tilde S=S\left[\frac{x_2}{x_1}\ddo\frac{x_d}{x_1}\right]
$$
and let $\bar S$ be the localization of $\tilde S$ at the prime ideal
$(x_1)\tilde S$:
$$
\bar S=\tilde S_{(x_1)}.
$$
The ring $\bar S$ is the local ring of the generic point of the exceptional divisor
of the blowing-up of $\s\ S$ at $\m_0$. Let $\bar R:=R\left[\frac{x_2}{x_1}\ddo\frac{x_d}{x_1}\right]$.
Then $\bar R$ is a semi-local integral domain, finite over $\bar S$ and
birational to $R$. Let $\q_0:=\m_0\bar S=(x_1)\bar S$; $\q:=\m\bar R$.

Let $\nu_0$ denote the restriction of $\nu$ to $\L$.
\begin{lemma}\label{lemma2.5}Let $(S,\m_0,\k)$ be a regular local ring with regular system of parameters
$$
x=(x_1\ddo x_d)
$$
and field of fractions $\L$. Let $\nu_0$ be a valuation of $\L$,  centered at $S$, such that $\bar x_1\ddo\bar x_d$ are algebraically independent in $\gr_{\nu_0}S$ over $\k$. Then the valuation $\nu_0$ is monomial with respect to $x$.
\end{lemma}
\begin{proof} In what follows, "monomial'' will mean a monomial in $x$ and "monomial ideal'' - an ideal, generated by monomials in $x$. For an element $f\in S$, let $M(f)$ denote the smallest (in the sense of inclusion) monomial ideal of $S$, containing $f$. The ideal $M(f)$ is well defined: it is nothing but the intersection of all the monomial ideals containing $f$. Let $Mon(f)$ denote the minimal set of monomials generating $I(f)$. In other words, $Mon(f)$ is the smallest  set of monomials in $x$ such that $f$ belongs to the ideal of $S$ generated by $Mon(f)$. The set $Mon(f)$ can also be characterized as follows. It is the unique set $\{\omega_1\ddo\omega_s\}$ of monomials, none of which divide each other and such that $f$ can be written as
\begin{equation}\label{2.3}
f=\sum\limits_{i=1}^sc_i\omega_i\quad\text{ with $c_i$ units of }S.
\end{equation}
A key point is the following: since $\bar x_1\ddo\bar x_n$ are algebraically independent in $\gr_{\nu_0}S$ by assumption, (\ref{2.3}) implies that
\begin{equation}\label{2.4}
\nu_0(f)=\min\limits_{1\leq i\leq s}\{\nu(\omega_i)\}.\end{equation}

If $I$ is an ideal of $S$, let $M(I)$ denote the smallest monomial ideal, containing $I$. The ideal $M(I)$ is generated by the set 
$\{Mon(f)\ |\ f\in I\}$.

We want to prove that all the $\nu_0$-ideals of $S$ are monomial. Let $I$ be a $\nu_0$-ideal of $S$. It is sufficient to prove that $I=M(I)$. Obviously, $I\subset M(I)$. To prove the opposite inclusion, take a monomial $\omega\in M(I)$. By the above, there exists $f\in I$ and $\omega'\in Mon(f)$ such that
\begin{equation}\label{2.5}
\left.\omega'\ \right|\ \omega.\end{equation}

By (\ref{2.4}), we have $\nu_0(\omega')\ge\nu_0(f)$. Since $I$ is a $\nu_0$-ideal, this implies that $\omega'\in I$. Hence $\omega\in I$ by (\ref{2.5}). This completes the proof.
\end{proof}
\begin{lemma}\label{lemma2.6}  Let $(S,\m_0,k)$ be a regular local ring with regular system of parameters
$$
x=(x_1\ddo x_d)
$$
and field of fractions $\L$. Let $\nu_0$ be a monomial valuation of  $\L$,  centered at $S$, such that the semigroup 
$\nu_0(S\setminus\{0\})$ is archimedian. Then the $\nu_0$-adic topology on $S$ is linearly equivalent to the $\m_0$-adic topology.
\end{lemma}
\begin{proof} Renumbering the $x_i$, if necessary, we may assume that $\nu_0(x_1)\le...\le\nu_0(x_d)$. Since $\nu_0(S\setminus\{0\})$ is archimedian, there exists a natural number $N$ such that $\nu_0(x_d)\leq N\nu_0(x_1)$. For $l\in\N$, let $Q_l$ denote the $\nu_0$-ideal
\begin{equation}\label{2.6}
Q_l:=\{x\in S\mid\nu_0(x)\geq l\nu_0(\m_0)\}.\end{equation}

Then for all $l\in\N$ we have $Q_{Nl}\subset\m_0^l\subset Q_l$. This completes the proof of the Lemma.
\end{proof}
Next, $R'$ be the normalization of $R$. Since $R$ is complete, $R'$ is a local ring \cite{11},
(37.9). Let $\m'$ be the maximal ideal of $R'$. If the $\m'$-adic
topology in $R'$ is linearly equivalent to the $\nu$-adic one then the same
is true for the $\m$-adic topology in $R$. Hence we may assume that $R$ is
normal.

Let $\K_1$ be a finite extension of $\K$ which is normal over $\L$  (in the sense of field theory).
Then there exists a valuation $\nu_1$ of $\K_1$ whose restriction to $\K$ is
$\nu$ \cite{26}, Chapter VI,  4,6,11. Let $R_1$ be the integral closure
of $R$ in $\K_1$. Then $R_1$ is a product of complete local rings and since it is an integral domain (it is a subring of $\K_1$) it is a complete local ring. Then $\nu_1$ is centered in the maximal ideal $\m_1$ of
$R_1$. We have $\dim(R_1)=\dim R$, $\rr\ \nu_1=\rr\ \nu$ and
$\left[\frac{R_{\nu_1}}{\m_{\nu_1}}:\frac{R_{\nu}}{\m_{\nu}}\right]<\infty$. The ring $R_1$ is
analytically irreducible by \cite{11}, (37.8). Finally, since
$R_1$ is algebraic over $R$, for any $x\in R_1$ there
exist $n\in\N$ and $y\in R$ such that $n\nu_1(x)\geq\nu_1(y)=\nu(y)$.
Hence $\nu_1$ is archimedian on $R_1$. Therefore $\nu_1$ satisfies (1) of
Theorem \ref{main_th}. To prove that the $\m$-adic topology on $R$  is linearly equivalent to the $\nu$-adic one, it is sufficient to prove that the same is true of the $\m_1$-adic topology on $R_1$. Thus we may assume that the field extension $\L\hookrightarrow \K$ is normal. Let $p=\ch\ \K$ if $\ch\ \K>0$ and $p=1$ otherwise. Let
$p^n$ denote the inseparability degree of $\K$ over $\L$.

By Lemmas \ref{lemma2.5}--\ref{lemma2.6} the $\nu_0$-adic topology on $S$ is linearly equivalent to the $\m_0$-adic topology.
Now, let $f\in\m$. Since $R$  is assumed to be integrally closed in $\K$, it equals the integral closure of $S$ in $\K$. Therefore $R$  is mapped to itself by all the automorphisms of $\K$ over $\L$. In particular, for every $\sigma\in Aut(\K/\L)$ we have 
$\nu(\sigma f)>0$. Then
$$
\nu(f)<\nu\left(\prod_{\sigma\in Aut(\K/\L)}\sigma f^{p^n}\right)=\nu(N_{\K/\L}(f))=\nu_0(N_{\K/\L}(f)).
$$
By the linear equivalence of topologies on $S$, there exists $r\in\N$,
such that
$$
\nu_0(N_{\K/\L}(f))\leq r\m_0(N_{\K/\L}(f)).
$$
Now Theorem \ref{main_th} follows from the next Lemma (we state Lemma \ref{lemma_norm} in greater generality than is necessary for Theorem \ref{main_th} for future reference).
\begin{lemma}\label{lemma_norm} Let $S\subset R$ be two noetherian domains with fields
of fractions $\L$  and $\K$, respectively. Let $\m$ be a maximal ideal of $R$
and $\m_0:=\m\cap S$. Assume that $S_{\m_0}$ and $R_\m$ are analytically
irreducible and that $R_\m$ is finite over $S_{\m_0}$. Assume that at least one of the following conditions holds:
\begin{enumerate}
\item for any $f\in R$, $N_{\K/\L}(f)\in S$.
\item The $\m_0$-adic order on $S$ is a valuation
\end{enumerate}
(so that the expression $\m_0(N_{\K/\L}(f))$ makes sense).
Then there exists $r\in\N$ such that for
any $f\in R$
$$
\m_0(N_{\K/\L}(f))\leq r\m(f).
$$
\end{lemma}
\begin{proof} Arguing as above, we reduce the problem to the case when $R$  and $S$ are complete local rings, $R$  is integrally closed in $\K$ and the field extension $\L\hookrightarrow\K$ is normal. We will work under all these assumptions from now on.

First, we prove the equivalence of topologies under the additional assumptions that $\K$ is separable (hence Galois) over $\L$  and $S$ is regular. We fix a regular system of parameters $x=(x_1\ddo x_d)$ of $S$. Let $G:=Gal(\K/\L)$. Then $G$ acts on $R$.

Let $\bar R$ be as above and let $\bar R'$ denote the integral closure of $\bar R$
in $\K$. Then $\bar R'$ is a 1-dimensional semi-local ring, finite over
$\bar S$. Let
$\m_1$,..., $\m_s$ denote the maximal ideals of $\bar R'$. Write
$$
\q\bar R'=\bigcap_{i=1}^s\m_i^{k_i}=\prod_{i=1}^s\m_i^{k_i}.
$$
Then
\begin{equation}\label{2.7}
\q^n\bar R'=\bigcap_{i=1}^s\m_i^{nk_i}=\prod_{i=1}^s\m_i^{nk_i}.\end{equation}

Each $\m_i$ defines a divisorial valuation of $\K$, centered in $R$. We
denote it by $\nu_i$. The group $G$ acts on $\bar R'$ and permutes the $\m_i$. By
Theorem \ref{Izumi}, all the $\nu_i$ are linearly comparable in $R$ to the $\m$-adic
pseudo-valuation of $R$. Hence there exists $r\in\N$ such that for any
$i$, $1\leq i\leq s$ and any $f\in R$,
\begin{equation}\label{2.8}
\nu_i(f)\leq r\m(f).
\end{equation}
Since $S$ is a UFD, $q_0(f)=\m_0(f)$ for all $f\in S$. Now, take any $f\in
R$. Without loss of generality, assume that
$$
\nu_1(f)=\max_{1\leq i\leq s}\nu_i(f).
$$
Finally, let $l\in\N$ be such that $\q^l\cap S\subset\q_0$. Then for
any $g\in S$, $\q_0(g)\leq l\q(g)$. We have

\begin{align}
\m_0(N_{\K/\L}(f))&=\q_0(N_{\K/\L}(f))\leq l\q(N_{\K/\L}(f))\leq\\
&\leq l\max_{1\leq i\leq s}k_i\max_{1\leq i\leq s}\nu_i(N_{\K/\L}(f))\\
&\leq l\max_{1\leq i\leq s}k_i\max_{1\leq i\leq s}\prod_{\sigma\in
G}\nu_i(\sigma f)\\
&\leq l\max_{1\leq i\leq s}k_i[\K:\L]\nu_1(f)\leq rl\max_{1\leq i\leq
s}k_i[\K:\L]\m(f).
\end{align}

This proves Lemma \ref{lemma_norm} in the case when $\K$ is Galois over $\L$  and $S$ is regular.

Continue to assume that $S$ is regular, but drop the assumption of separability of $\K$ over $\L$.
Let $\K_s$ denote the maximal separable extension of $\L$  in $\K$. Let
$R_s:=R\cap\K_s$. Suppose $\ch\ \K=p>0$. Then there exists $n\in\N$ such
that $R^{p^n}\subset R_s$. Let $\m_s$ denote the maximal ideal of $R_s$.
Since $\m_s\subset\m$, for any $g\in R_s$ we have $\m_s(g)\leq\m(g)$. By the separable case there exists $r\in\N$ such that for any $f\in R_s$
$$
\nu(f)\leq r\m_s(f).
$$
Then for any $f\in R$,

\begin{align}
\nu(f)&=\frac{1}{p^n}\nu(f^{p^n})\leq\frac{r}{p^n}\m_s(f^{p^n})\\
&\le\frac{r}{p^n}\m(f^{p^n})\leq r\m(f).
\end{align}

Hence $\nu(f)$ and $\m(f)$ are linearly comparable, as desired. This proves Lemma \ref{lemma_norm} assuming $S$ is regular.

Finally, drop the assumption that $S$ is regular. There exists a complete regular local ring $T\subset S$
such that $S$ is finite over $T$. Since Lemma \ref{lemma_norm} is already known for $T$, it is also true for $S$ by the multiplicativity of the norm. This proves Lemma \ref{lemma_norm} and with it Theorem \ref{main_th}.
\end{proof}


\section{Applications}\label{S3}

The rest of the paper is devoted to the applications of Izumi's Theorem and of Lemma \ref{lemma_norm}. The first
application is to rewrite some of the classical theorems on comparison of
topologies in noetherian rings (which were traditionally proved by
Chevalley lemma) to include linear equivalence of topologies. \\

The following observation will be useful in the sequel.
\begin{lemma}\label{lemma_loc} Let $R$  be a noetherian ring, $\m$ a maximal ideal of $R$. Let $\phi:R\rightarrow R_{\m}$ denote the
localization homomorphism. Then
$$
\phi^{-1}(\m^nR_\m)=\m^n+\operatorname{Ker}\phi=\m^n.
$$
In other words, the symbolic powers of $\m$ coincide with the usual powers. In particular, the $\m$-adic topology on $R$  coincides with the restriction of the $\m^nR_\m$-adic topology to $R$.
\end{lemma}
\begin{proof} Consider an element
$$
x\in\phi^{-1}(\m^nR_\m).
$$
Then there exists $u\in R\setminus\m$ such that $ux\in\m^n$. For every natural number $n$, we have $(u)+\m^n=R$. Then there exist 
$v_n\in R$, $m_n\in\m^n$ such that $uv_n+m_n=1$. We have $x=x\cdot1=x(uv_n+m_n)=xuv_n+xm_n\in\m^n$. This proves that 
$\phi^{-1}(\m^nR_\m)\subset\m^n$ for all natural numbers $n$, and the Lemma follows.
\end{proof}
The following Corollary is a partial generalization of  Corollary 2, \cite{26} p. 273:
\begin{corollary} \label{Zar} Let $R$ be a noetherian ring and $\m$ a maximal ideal of
$R$, such that $R_\m$ is an analytically irreducible local domain.  Let
$\bar R$ be a finitely generated $R$-algebra, containing $R$. Let $\p$ be a
prime ideal of $\bar R$, lying over $\m$. Then the $\p$-adic topology on $R$
is linearly equivalent to the $\m$-adic topology. In other words, there
exists $r\in\N$ such that for any $n\in\N$
$$
\p^{rn}\cap R\subset\m^n.
$$
\end{corollary}
\begin{proof} By Lemma \ref{lemma_loc}, we may assume that $R$ is an analytically irreducible local noetherian
domain with maximal ideal $\m$. Let $\K$ be the field of fractions of $R$.
Let $\psi:R\rightarrow\wdh R$ denote the $\m$-adic completion of $R$. The homomorphism
$\psi$ is faithfully flat, hence so is the induced map $\bar
R\rightarrow\bar R\otimes_R\wdh R$. Then there exists a prime ideal
$\wdh\p$ in $\bar R\otimes_R\wdh R$ which lies over $\p$. The ring $\bar
R\otimes_R\wdh R$ is finitely generated over $\wdh R$, so we may assume that $R$
is $\m$-adically complete.

If $\bar R$ is a purely transcendental extension of $R$, then $\p^n\cap
R=\m^n$ and the Corollary is trivial.

The normalization $R'$ of $R$ is a complete local domain, finite over
$R$ (\cite{11}, Corollary 37.9). Hence the Corollary is true when $\bar R=R'$.
Replacing $(R,\bar R)$ with $(R',R'\otimes_R\bar R)$, we may assume that $R$
is normal, hence analytically normal.

Since we know the Corollary for the case of purely transcendental
extensions, we may replace $R$ by the normalization of a maximal purely
transcendental extension of $R$, contained in $\bar R$. Thus we may assume
that the total ring of fractions  of $\bar R$ is finite over $\K$.

Let $\q$ be a minimal prime of $\bar R$ such that $\q\subset\p$. Since $R$
is a domain, $\p\cap R=(0)$. Replacing $\bar R$ with $\frac{\bar R}\q$, we
may assume that $\bar R$ is a domain. Let $\bar \K$ denote its field of
fractions. Let $\pi:X\rightarrow\text{Spec}\ \bar R$ be the normalized blowing-up along
$\p$. Since $R$ is Nagata, so is $\bar R$ and $\pi$ is of finite type. Let
$E$ be any irreducible component of $\pi^{-1}(\p)$. Let $\nu$ denote the
divisorial valuation of $\bar \K$ associated to $E$, and let
$$
\p_l:=\{x\in\bar R\mid\nu(x)\geq l\},\qquad l\in\N.
$$
Then $\p^l\subset\p_l$ for all $l\in\N$. Hence it is sufficient to show
that the $\nu$-adic and the $\m$-adic topologies in $R$ are linearly
equivalent. Since $[\bar \K:\K]<\infty$, $\nu$ induces a divisorial valuation of
$\K$, centered in $R$. Now the Corollary follows from Theorem \ref{main_th}.
\end{proof}
The above Corollary can be strengthened as follows.
\begin{corollary}\label{cor_trans}
Let $R$ be a noetherian ring and $\m$ a maximal ideal of $R$ such that $R_\m$
is an analytically irreducible local domain. Let $\bar R$ be a finitely
generated $R$-algebra, containing $R$. Let $\p$ be a prime ideal of $\bar
R$ lying over $\m$ and $I$ an ideal of $\bar R$ such that $I\subset\p$ and
$I\cap R=(0)$. Then there exists $r\in\N$ such that for any $n\in\mathbb
N$,
$$
(I+\p^{rn})\cap R\subset\m^n.
$$
\end{corollary}
\begin{proof} Replacing $\bar R$ by $\frac{\bar R}I$ does not change the problem. Now Corollary \ref{cor_trans} follows from Corollary
\ref{Zar}.
\end{proof}
\noindent\begin{rmk} In terms of the analogy with functional analysis, this Corollary
says that if $I\cap R=(0)$, then $I$ is ``transversal'' to $R$.
\end{rmk}
\begin{corollary}\label{cor_fact}Let $R$ be a noetherian ring and $T=(T_1\ddo T_n)$
independent variables. Let $\m$ be a maximal ideal of $R$. Let $\bar R$ be a finitely generated extension of $R$ and $\bar 
p\subset\bar\m\subset\bar R[[T]]$ a pair of prime ideals of $\bar R[[T]]$ such that $\bar\m\cap R[[T]]=(\m,T)$. Let 
\begin{equation}
\p:=\bar p\cap R[[T]]\subset(\m,T).\label{eq:pdanspbar}
\end{equation}
Assume that $\frac{R[[T]]_{(\m,T)}}\p$ is analytically
irreducible. Then there exists $r\in\N$ such that
for all $n\in\N$
$$
(\bar\m^{rn}+\bar\p)\cap R[[T]]\subset(\m,T)^n+\p.
$$
\end{corollary}
\noindent\begin{rmk} In particular, we can apply this Corollary to the following situation. Let $R$, $T$, $\m$, $\bar R$, $\bar \m$ be as in Corollary \ref{cor_fact}. Assume, in addition, that $R$ is a UFD. Let $\p=(F)$ be a principal prime ideal generated by a single irreducible power series $F\in R[[T]]$. Let
\begin{equation}\label{3.1}
\p\bar R[[T]]=\q_1\cap\cdots\cap\q_s
\end{equation}

be a primary decomposition of $\p$ in $\bar R[[T]]$. We must have $\sqrt{\q_i}\subset\bar\m$ for some $i$. Let
$\bar\p:=\sqrt{\q_i}$. Finally, assume that the equality (\ref{eq:pdanspbar}) is satisfied.
The equality \eqref{3.1}
corresponds to a factorization $F$ in $\bar R[[T]]$. Say, $F=F_1F_2$ in
$\bar R[[T]]$. Then Corollary \ref{cor_fact} says that there exists $r\in\N$ such
that if
$$
\tilde F\equiv F_1\tilde F_2\quad\md\ \bar\m^{rn},
$$
where $\tilde F_2\in\bar R[[T]]$, $\tilde F\in R[[T]]$, then
$$
\tilde  F\cong Fg\quad\md\ (\m,T)^n
$$
for some $g\in R[[T]]$. Thus $F_1(F_2g-\tilde F_2)\in (\m,T)^n\subset \tilde \m^n$ and by Artin-Rees Lemma there exists a constant $c$ depending only on $F_1$ such that $\tilde F_2\cong F_2g$ $\md$
 $\bar \m^{n-c}$. In other words, approximate factorization of an
element of $R[[T]]$ in $\bar R[[T]]$ is close to an actual factorization,
and the estimate is linear in $n$.\end{rmk}

\begin{proof}[Proof of  Corollary \ref{cor_fact}] Let $\tilde R:=\bar R[[T]]_{\bar \m}$ and let
$$
S=R[[T]]_{(\m,T)}\otimes_R\bar R_{\bar \m\cap\bar R}\subset\tilde R.
$$
Let $\m_0:=\bar\m\tilde R\cap S$. Then $\tilde R$ is the $(T)$-adic
completion of $S$, followed by localization at $\bar \m$. We can decompose the injective homomorphism
$S\rightarrow\tilde R$ in two steps: $S\rightarrow
S_{\m_0}\rightarrow\tilde R$, where the second arrow is a faithfully flat
homomorphism and the first a localization with respect to a maximal ideal. By faithful flatness, we have $\bar \m^n\cap S_{\m_0}=\m_0^n\tilde R\cap S_{\m_0}=\m_0^nS_{\m_0}$, so that the $\bar \m$-adic topology on $S_{\m_0}$ is linearly equivalent to the $\m_0$-adic topology. Combining this with Lemma \ref{lemma_loc}, we see that the $\bar \m$-adic topology on $S$ is linearly equivalent to the $\m_0$-adic topology. The ring $S$ is a localization of a finitely generated $R[[T]]$-algebra, and we
apply Corollary \ref{Zar} to the ring extension
$\frac{R[[T]]}\p\hookrightarrow\frac S{\bar\p\cap S}$. This completes the
proof.
\end{proof}

\noindent\begin{rmk} Corollary \ref{cor_fact} can be strengthened as follows. Let $R$ be
a noetherian ring, $T_1\ddo T_n$ independent variables and $A$ a noetherian
ring such that
$$
R[T]\subset A\subset R[[T]].
$$
Let $\m$ be a maximal ideal of $R$. Let $\bar R$ be
a finitely generated $R$-algebra and let $B$ be a noetherian $\bar
R$-algebra such that
$$
A\otimes_R\bar R\subset B\subset\bar R[[T]].
$$
Assume that $(\m,T)$ is a maximal ideal of $A$ and that both
$A\otimes_R\bar R$ and $B$ have $\bar R[[T]]$ as their
$(T)$-adic completion. Let $\bar\m$ be any prime ideal of $\bar R$ lying over
$(\m,T)$ and $\bar\p$ a prime ideal of $B$ such that $\bar\p\subset\bar\m$.

Let $\p:=\bar\p\cap A$. Assume that $\frac{A_{\bar\m\cap A}}{\p A_{\bar\m\cap A}}$ is analytically irreducible. Then there exists 
$r\in\N$ such that for any $n\in\N$
$$
(\bar \m^{rn}+\bar\p)\cap A\subset(\m,T)^n+\p.
$$
The proof is exactly the same as for Corollary \ref{cor_fact}.
\end{rmk}
\vskip0.1in\noindent

The following result is not a corollary of Theorem \ref{main_th} but uses Theorem \ref{Rees_val}:
\begin{prop}(cf. \cite{11} Theorem 3.12, p. 11). Let $R$ be a
noetherian domain and $I$ an ideal of $R$. Let $x$ be a non-zero element of
$R$. Assume that $R$ is analytically unramified. Then there exists $r\in\N$ such that
for any $k,n\in\N$
$$
I^k:x^n\subset I^{k-rn}.
$$
Here we adopt the convention that $I^n=R$ if $n\le0$.
\end{prop}
\begin{proof} By Theorem \ref{Rees_val}, there exist valuations $\nu_1\ddo\nu_s$ such that for any $f\in R$
$$
\bar I(f)=\min_{1\leq i\leq s}\frac{\nu_i(f)}{\nu_i(I)}.
$$
By Theorem \ref{Rees_th} there exists $r_1\in\N$ such that for any $n\in\N$ and any $f\in R$
\begin{equation}\label{3.2}
\bar I(f)\leq r_1+I(f).\end{equation}

Hence, for any $l\in\N$, if $\bar I(f)\geq r_1+l$ then $f\in I^l$. For any $y\in I^k:x^n$ we must have
$$
\nu_i(y)+n\nu_i(x)\geq k\nu_i(I)\qquad\text{for all }i,\ 1\leq i\leq s.
$$
Now take a positive integer $R$  such that
$$
r\ge\max_{1\leq i\leq s}\frac{\nu_i(x)}{\nu_i(I)}+r_1.
$$
Then for any $y\in I^k:x^n$ and any $i\in\{1\ddo s\}$ we have
$$
\nu_i(y)\geq k\nu_i(I)-n\nu_i(x)\geq k\nu_i(I)-n(r-r_1)\nu_i(I)\geq(k-nr+r_1)\nu_i(I).
$$
By (\ref{3.2}) this implies that $y\in I^{k-nr}$, as desired.
\end{proof}

The following corollary is a generalization of the main result of \cite{M}:

\begin{corollary} Let $R$ be an analytically irreducible noetherian local ring. Then there exists $a\in \N$ such that for any  proper ideal $I$ of $R$  we have:
$$
I^{(ac)}\subset \m^c \ \ \forall c\in \N.
$$
Here, if $W$ denotes the complement of the union of the associated primes of $I$, $I^{(n)}$ is the contraction of $I^nR_W$ to $R$ where $R_W$ denotes the localization of R with respect to the multiplicative system $W$. The ideal $I^{(n)}$ is called the $n$-th symbolic power of $I$.
\end{corollary}
Let us mention that it is known that if $(R,\m)$ is a regular local ring of dimension $d$ then $I^{(dc)}\subset I^c$ for any ideal $I$ of $R$  and any integer $c$ \cite{5}. If $(R,\m)$ has is an isolated singularity ring then $I^{(kc)}\subset I^c$ for any ideal $I$ of $R$   and any integer $c$ \cite{6} for some constant $k$ independent on $\p$. For a general local ring $R$  and for any ideal $I$, there exists a constant $k$ depending on $I$ such that $I^{(kc)}\subset I^c$ for any $c$  \cite{Sw} but it is still an open question to know if  such a $k$ may be chosen independently of $I$ in general.
\begin{proof}
First let us prove the result when $R$  is a complete local domain and $I=\p$ is a prime ideal.
By Cohen's structure theorem $R$  is finite over a ring of power series over a field or over a compete Dedekind domain $S$. We denote by $\m_0$ the maximal ideal of $S$. Let $\K$ (resp. $\L$ ) denote the field of fractions of $R$  (resp. $S$).\\
\\
 First let us assume that $\K/\L$ is Galois and $R$  is normal.
Let $\q:=\p\cap S$. Since $\K /\L$ is Galois and $R$  is the integral closure of $S$ in $\K$, there exists an integer $l$ which is independent of $\p$ such that for any integer $N$, $x\in \p^{(N)}$ implies $N_{\K/\L}(x)^l\in \q^{(N)}$ (see Proposition 3.10 of 
\cite{4}). Since $S$ is a regular local ring  we have $\q^{(k)}\subset \m_0^k$ for any $k$ (see \cite{4} p. 9). By Lemma 
\ref{lemma_norm} there exists an integer $r\in\N$ such that for any $f\in R$, 
$\m_0(N_{\K/\L}(f))\leq r\m(f)$. Thus if $x\in\p^{(rkl)}$ then $N_{\K/\L}(x)^l\in \m_0^{rkl}$, hence $N_{\K/\L}(x)\in \m_0^{rk}$ and we have $x\in \m^{ k}$. Finally we obtain
$$
\p^{(rl c)}\subset \m^c\ \ \ \forall \p\subset R \text{ prime  and }c\in\N.
$$
\\
Next, keep the assumptions that $R$ is complete and $\p$ is prime, but drop the assumptions that $R$ is normal and that the extension
$\L\rightarrow\K$ is Galois. Let $p=\ch\  \K$ if $\ch\ \K >0$ and $p=1$ otherwise. Let $p^n$ be the inseparability degree of $\K$ over $\L$. Let $\K_s$ be the maximal separable extension of $\L$ in $\K$ and set $R_s:=R\cap \K_s$. Then $R_s$ is a complete local domain whose maximal ideal $\m_s$ equals $\m\cap R_s$ and $R^{p^n}\subset R_s$. The ideal $\p_s:=\p\cap R_s$ is a prime ideal of $R_s$ and 
$\p^{p^n}\subset \p_s$. For any element $y\in\p^{(c)}$ there exists $a\in R\backslash \p$ such that $ay\in \p^c$. Thus 
$a^{p^n}y^{p^n}\in \p_s^c$ and $y^{p^n}\in{\p_s^{(c)}}$. 

If $\p_s^{(\alpha c)}\subset \m_s^c$ for any integer $c$, then for any
$x\in\p^{(\alpha p^n c)}$ we have $x^{p^n}\in \p_s^{(\alpha p^n c)}\subset \m_s^{p^n c}$. Thus $x^{p^n}\in\m^{p^nc}$ and by Rees theorem there exists a constant $c_0$ depending only on $R$  such that $x\in\m^{c-c_0}$. Thus we may assume that $\K/\L$ is separable.\\
In this case let us denote by $\K_1$ a finite separable field extension of $\K$ which is normal over $\L$  and let $R_1$ be the integral closure of $R$  in $\K_1$. Then $R_1$ is a direct sum of complete local rings and since $R_1$ is a  domain (it is a subring of a field) it is a complete local domain. Let $\m_1$ be the maximal ideal of $R_1$. By Lemma 2.4 \cite{R} there exists $\alpha\in \N$ such that $\m_1^{\alpha c}\cap R\subset \m^c$ for any integer $c$. Since $R\rightarrow R_1$ is finite there exists a prime ideal $\p_1$ of $R_1$ lying over $\p$. Thus by  replacing $R$  and $\p$ by $R_1$ and $\p_1$, we may assume that $\K/\L$ is Galois and $R$  is normal and this case has been proved above.\\
\\
Now let us assume that $R$ is an analytically irreducible local ring and $I=\p$ is a prime ideal of $R$. Let us consider an irredundant primary decomposition of $\p\wdh{R}$:
$$\p\wdh{R}=\q_1\cap\cdots\cap \q_s$$
where $\wdh{R}$ denotes the completion of $R$ and the $\q_i$ are primary ideals of $\wdh{R}$. Let $\p_i$ be the radical of $\q_i$ for all $i$ and set $W:=\wdh{R}\backslash \cup_i\p_i$. Since $\p_i\cap R=\p$ for any $i$, we have an inclusion of multiplicative systems:
$R\backslash\p\subset W$. Thus we have for any integer $n$:

$$
\p^{(n)}=\p^nR_{R\backslash\p}\cap R\subset\p^n\wdh{R}_{R\backslash\p}\cap 
R\subset\p^n\wdh{R}_W\cap\wdh{R}\subset\p_1^{(n)}\wdh{R}_{\wdh{R}\backslash \p_1}\cap\wdh{R}=\p_1^{(n)}.
$$
By the previous case there exist $a$ and $b$ such that $\p_1^{(ac+b)}\subset\wdh{\m}^c$ for all integers $c$. Since $\wdh{\m}^c\cap 
R=\m^c$ the theorem is proved in this case.\\
\\
Finally let us assume that $R$ is an analytically irreducible local ring and $I$ is any ideal of $R$, not necessarily prime. Let 
$\p_1$,..., $\p_s$ be the associated primes of $I$ and set $W=R\backslash\cup_ip_i$. Then we have for any integer $n$:
$$I^{(n)}=I^nR_W\cap R\subset \p_1^nR_W\cap R\subset \p_1^nR_{R\backslash\p_1}\cap R=\p_1^{(n)}.$$
Since the theorem is proved for the symbolic powers of $\p_1$, this proves the theorem for any ideal $I$.\\

\end{proof}


\end{document}